\documentclass[conference]{IEEEtran}
\IEEEoverridecommandlockouts

\usepackage{booktabs} 
\usepackage{nicefrac}  
\usepackage{microtype}
\usepackage[utf8]{inputenc}

\usepackage[dvipsnames]{xcolor}
\usepackage{wrapfig}
\usepackage{multirow}

\usepackage{amssymb}
\usepackage{amsthm}
\usepackage{bbm}
\usepackage{array}
\usepackage[colorlinks=true,urlcolor=purple,
linkcolor=purple,citecolor=purple]{hyperref}
\usepackage{amsmath}
\usepackage{mathtools}
\usepackage{float}

\usepackage{mathrsfs}
\usepackage{graphicx}
\usepackage{color}

\usepackage{algorithm}
\usepackage{algpseudocode}
\usepackage{thmtools,thm-restate} 

\makeatletter
\def\paragraph{\@startsection{paragraph}{4}%
  \z@\z@{-\fontdimen2\font}%
  {\normalfont\bfseries}}
\makeatother

\usepackage{tikz}
\usetikzlibrary{decorations.markings, backgrounds, calc, positioning, shapes, shadows, arrows, fit, automata}
\usepackage{tikz-cd}
\usetikzlibrary{arrows.meta}
\tikzset{>={Latex}}

\newcommand{\Z}{\mathbb{Z}}

\newcommand{\mc}{\mathcal}
\newcommand{\mb}{\mathbb}

\renewcommand{\b}{\beta}
\newcommand{\g}{\gamma}
\renewcommand{\d}{\delta}

\newcommand{\s}{\sigma}

\newcommand{\p}{\partial}
\renewcommand{\t}{\tau}

\newcommand{\Om}{\Omega}

\newcommand{\Si}{\Sigma}
\newcommand{\La}{\Lambda}

\newcounter{desccount}

\newcommand{\descref}[1]{\hyperref[#1]{#1}}

\newcommand{\lp}{\left(}
\newcommand{\rp}{\right)}

\newcommand{\set}[1]{\left\{#1\right\}}

\renewcommand{\r}{\rightarrow}
\newcommand{\xr}{\xrightarrow}

\newcommand{\im}{\operatorname{im}}

\newcommand{\pathhom}{\mathbf{PathHom}}
\newcommand{\dfchom}{\mathbf{DFCHom}}

\newcommand{\mlp}[1]{K^{\rightarrow}_{n_1,\ldots,n_{#1}}} 
\newcommand{\mlpgraph}[1]{K_{n_1,\ldots,n_{#1}}}

\newcommand{\rank}{\operatorname{rank}}

\newtheorem{theorem}{Theorem}

\newtheorem{proposition}[theorem]{Proposition}
\newtheorem{lemma}[theorem]{Lemma}

\theoremstyle{definition}

\makeatletter
\newcommand{\pushright}[1]{\ifmeasuring@#1\else\omit\hfill$\displaystyle#1$\fi\ignorespaces}
\newcommand{\pushleft}[1]{\ifmeasuring@#1\else\omit$\displaystyle#1$\hfill\fi\ignorespaces}
\makeatother

\title{Path homologies of deep feedforward networks\\
\thanks{}
}

 \author{\IEEEauthorblockN{Samir Chowdhury}
 \IEEEauthorblockA{\textit{Stanford University}\\
 samirc@stanford.edu}
 \and
 \IEEEauthorblockN{Thomas Gebhart}
 \IEEEauthorblockA{\textit{University of Minnesota}\\
 gebhart@umn.edu}
 \and
 \IEEEauthorblockN{Steve Huntsman}
 \IEEEauthorblockA{\textit{BAE Systems FAST Labs} \\
 steve.huntsman@baesystems.com}
 \and
 \IEEEauthorblockN{Matvey Yutin}
 \IEEEauthorblockA{\textit{UC Berkeley and BAE Systems FAST Labs} \\
 myutin@berkeley.edu}
 }

\begin{document}
\maketitle

\begin{abstract}
We provide a characterization of two types of directed homology for fully-connected, feedforward neural network architectures. These exact characterizations of the directed homology structure of a neural network architecture are the first of their kind. We show that the directed flag homology of deep networks reduces to computing the simplicial homology of the underlying undirected graph, which is explicitly given by Euler characteristic computations. We also show that the path homology of these networks is non-trivial in higher dimensions and depends on the number and size of the layers within the network. These results provide a foundation for investigating homological differences between neural network architectures and their realized structure as implied by their parameters.
\end{abstract}

\section{Introduction}

Deep neural networks have emerged as an effective class of machine learning algorithms across a wide range of domains \cite{lecun2015deep}. The superior performance of these algorithms in comparison to other traditional machine learning methods may be attributed to their structure, consisting of a layer-wise composition of nonlinear functions parameterized by a set of real-valued weight matrices which define connectivity between layers \cite{hornik1991approximation, mhaskar2017and}.

These parameters, governing connectivity and information flow within the network, can then be optimized by gradient descent for performance on a given task. For certain tasks, enforcing particular connectivity patterns by altering the network architecture can benefit the performance of the algorithms. These architectural priors are relevant to the functionality of deep neural networks on particular tasks as they define \textit{a priori} a partitioning of the input information as it is processed throughout the network. Because the architecture is fixed during training, gradient descent searches for an optimal parameterization for the task given this connectivity structure. 

While certain connectivity biases have proven successful in various domains \cite{krizhevsky2012imagenet, he2015deep}, the problem of defining an optimal connectivity pattern for a given task is still an open one. Neural architecture search \cite{elsken2019neural} attacks this problem by attempting to optimize over network architectures using a variety of search strategies. However, the space of possible architectures is combinatorially large. This fact, combined with the high computational costs of training a single architectural instantiation on a task, means these search methods are severely restricted in their ability to properly span the possible space of network architectures. Better topological priors on optimal connectivity for a given task are needed to better constrain this search space. Recent work has shown that even in popular network architectures that achieve near-optimal performance on a task, a substructure with a drastically reduced parameter set that achieves similar task performance is likely to exist \cite{frankle2018lottery}. This result, in combination with the effectiveness of pruning techniques \cite{han1506learning, li2016pruning}, implies that many neural networks are overparameterized by architectures that do not properly constrain the partitioning of the input across the network. However, gradient descent is still able to approximate this optimal topology, and this is reflected within the trained parameters of the network \cite{frankle2018lottery}. A method for determining the extent to which network-topological structures of a trained network differ from the topological structure implied by its architecture can provide insight on the proper connectivity structure for a given task. We provide an initial step in this direction with the characterization of two directed homologies of fully-connected neural networks. This framework provides the means to compare the actualized connectivity structure of parameterized neural networks to the homological structure defined by their architectures.

Understanding the topological structure of the network architecture naturally translates to a question of understanding the topological structure of the underlying directed graph. Recent developments in this direction with strong mathematical foundations include the theories of \emph{path homology} and \emph{directed flag complex (DFC) homology}. Path homology was developed by Grigor'yan, Lin, Muranov, and Yau \cite{Grigoryan2012}, and an accompanying theory of digraph homotopy was later developed in \cite{Grigoryan2014b}. This in turn was consistent with earlier notions of homotopy of graphs \cite{babson2006homotopy}. DFC homology was popularized via \cite{reimann2017cliques}, and is built on top of a notion of ordered simplicial homology that has concrete mathematical foundations \cite{munkres-book}. In both of these cases, persistent-homological frameworks have been developed recently \cite{Chowdhury2018, turner2019rips}. These recent developments have thus provided novel tools for approaching the problem of understanding neural architecture.

\subsection{Contributions and statement of results}
\label{result}

In this paper, we provide a characterization of the reduced path homology of fully-connected, feedforward neural networks (i.e. multilayer perceptrons (MLPs)) in terms of their architecture. This exact characterization of the path homology structure of a neural network is the first of its kind. Additionally, we provide a characterization of the DFC homology of MLPs, and show that it reduces to computing simplicial homology of the underlying undirected graph viewed as a simplicial complex. These results provide a starting point for investigating differences between the inherent homological structure of a neural network architecture versus its realized (persistent) homological structure as implied by its learned parameters.

Specifically, let $\mlp{L}$ denote the directed acyclic graph corresponding to the architecture of an MLP with $L$ layers of widths $\set{n_1,n_2,\ldots, n_L}$. See Figure \ref{fig:IrisMLP} for an example.
Individual layers (as sets of vertices) are denoted $K_1,K_2,\ldots, K_L$. 
In graph-theoretic terms, $\mlp{L}$ has node set $\cup_{i=1}^L K_i$ and edge set $\{(v,v') : v \in K_i, \, v' \in K_{i+1}, \, 0\leq i \leq L-1\}.$ 
Also let $\mlpgraph{L}$ denote the underlying undirected graph, viewed as a simplicial complex. We show that $\mlp{L}$ has nontrivial reduced path homology (with field coefficients) precisely in degree $(L-1)$, and that this homology group has rank $\prod_{i=1}^L(n_i -1 )$. 

\begin{theorem}
\label{thm:phmlp}
Let $\mlp{L}$ be the MLP with $L$ layers of widths $\set{n_1,n_2,\ldots, n_L}$. Then we have
\[  \rank \lp\pathhom_p(\mlp{L} )\rp = \d^{L-1}_p \prod_{i = 1}^L (n_i - 1).\] 

\end{theorem}

Additionally, we show that the DFC homology of $\mlp{L}$ reduces to the simplicial homology (denoted $H_p^\Delta$) of $\mlpgraph{L}$, which in turn counts the number of loops in $\mlpgraph{L}$.   
\begin{theorem}
\label{thm:dfchmlp}
We have
\[  \rank \lp\dfchom_p(\mlp{L} )\rp = \rank \lp(H_p^\Delta(\mlpgraph{L}) \rp).\] 
Specifically, this rank is 1 for $p=0$, $(1-\#V+ \#E)$ for $p=1$, and 0 for $p \geq 2$. Here $\#V$ and $\#E$ are the numbers of vertices and edges, respectively.  
\end{theorem}
Stated differently, Theorem \ref{thm:dfchmlp} shows that $\dfchom$ picks out the structure of (undirected) loops in the MLP architecture.

\section{Related Work}

Work on homological approaches to neural network analyses have shown promise in the ability to extract insights about network function through the investigation of topological structure of network parameters \cite{carlsson2018topological, Nguyen2019, Gabella2019, bruel2019topology, guss2018characterizing, hofer2017deep, moor2019topological, hofer2019connectivity}. In \cite{Rieck2019}, the authors find that analyzing the $0$-dimensional persistent-homological structure of neural network parameters during training provides insight into when the training process may be considered completed. Similarly, the authors of \cite{Gebhart2019} analyze the $0$-dimensional persistent-homological structure of neural network activations and find that this structure is closely linked to the representations used by the network to make classification decisions. Both of the above papers make use of a Vietoris-Rips filtration over the graph defined by the network architecture. Notably, this filtration does not preserve the inherent directionality of the network defined by sequential layers. The asymmetry arising from directionality within the network is important to capture in order to faithfully represent the homological properties of the flow of information through each layer. The papers also lack a homological basis from which to compare the persistent homological structures that emerge as weights are thresholded to the inherent network homology. Our result in Section \ref{result} provides such a homological basis for path homology of fully-connected neural networks from which one may compare the actualized homological structure of the weighted network versus the homological structure defined by its architecture.  

The prior discussion relates to investigating the topological structure of the network architecture; there have also been very recent developments showing how to backpropagate a topological loss function through a deep neural network \cite{bruel2019topology, hu2019topology}. These techniques build on insights developed throughout \cite{chen2018topological, poulenard2018topological, gabrielsson2018topology}.
While our methods are grounded in constructions arising from topology, a geometric viewpoint of ReLU deep networks has been introduced in \cite{lei2018geometric}.

\section{Preliminaries}

In this section, we provide background material on homology, path homology, and DFC homology. We refer the reader to \cite{munkres-book} for additional details on homology (specifically \S1.13 for ordered homology), \cite{reimann2017cliques} for details on DFC homology, and \cite{Grigoryan2014b} for details on path homology. 

We write $\Z_+$ to denote the nonnegative integers. Fix a field $\mb{K}$. A \emph{chain complex} is defined to be a sequence of vector
spaces $(C_p)_{p\in \Z}$ over $\mb{K}$ and \emph{boundary maps} $(\p_p:C_p\r C_{p-1})_{k\in \Z}$ satisfying the condition $\p_{p-1}\circ \p_p =0$ for each $p \in \Z$. We often denote a chain complex as $\mc{C}=(C_p,\p_p)_{p\in \Z}$. Given a chain complex $\mc{C}$ and any $p\in \Z$, one defines the following: 
\begin{align*}
Z_p(\mc{C}) &:= \ker(\p_p) = \set{c\in C_k : \p_p(c) = 0}, \text{ the \emph{$p$-cycles}},\\ 
B_p(\mc{C}) &:=\im(\p_{p+1}) = \{c \in C_{p} : c = \p_{p+1} (b) \text{ for some }\\ 
& \hspace{1in} b\in C_{p+1}\}, \text{the \emph{$p$-boundaries}}.
\end{align*}

The quotient vector space $H_p(\mc{C}) :=Z_p(\mc{C})/B_p(\mc{C}) $ is called the \emph{$p$-th homology vector space of the chain complex $\mc{C}$}. The dimension of $H_p(\mc{C})$ is called the \emph{$p$-th Betti number} of $\mc{C}$, denoted $\b_p(\mc{C})$. These vector spaces can be made to arise from data via the following construction.

A \emph{simplicial complex} $\Sigma$ built on a set $S$ is (abstractly) a collection of subsets $\s \subseteq S$ such that whenever $\t \subseteq \s \in \Si$, we have $\t \in \Sigma$. The $(p+1)$-length elements of $\Si$ are referred to as \emph{$p$-simplices}. The elements of a simplex are called \emph{vertices}. We additionally fix an arbitrary total ordering on $S$. Different orderings of the vertices of a simplex are considered equivalent if they differ by an even permutation. Thus a $p$-simplex $\s \in S$ for $p \geq 1$ belongs to two equivalence classes, and each class is called an \emph{orientation} of $\s$. For each $p \in \Z_+$, we write $\Sigma_p$ to denote the $p$-simplices of $\Si$. 

The standard construction of a chain complex from a simplicial complex is obtained by defining $C_p(\Si)$ to be the free vector space over $\Si_p$ for each $p \geq 0$, with coefficients in $\mb{K}$, along with the relation $\s = -\t$ if $\t$ differs from $\s$ by an odd permutation. Additionally one defines $C_{-1}(\Si) = \mb{K}$ and $C_{p}(\Si) = \{0\}$ for $p \leq -2$ (this corresponds to \emph{reduced} homology). Finally, for any $p \in \Z_+ $, one defines a linear map $\p_p: C_p \r C_{p-1}$ to be the linearization of the following map on the generators of $C_p$:
\begin{align}
\p_p([x_0,\ldots,x_p]) &:=\sum_{i=0}^p(-1)^i[x_0,\ldots,\widehat{x_i},\ldots,x_p], \label{eq:bd-map}
\end{align}
for each $p$-simplex $[x_0,\ldots,x_p]\in C_p$. Here $\widehat{x_i}$ denotes omission of $x_i$ from the sequence.
Additionally, $\p_{p}$ is defined to be the zero map for $p \leq -1$. These constructions fully determine simplicial homology, which we denote by $H_p^\Delta$. 

An undirected graph $G = (V,E)$ has a natural representation as a simplicial complex: the nodes are 0-simplices, and the edges are 1-simplices. There are no higher-dimensional simplices. We write $\Si_p(G)$ to denote the $p$-simplices of $G$.

\subsection{Path homology}

Given a finite set $X$ and any integer $p \in Z_+$, an \emph{elementary $p$-path over $X$} is a sequence $(x_0,\ldots, x_p)$ of $p+1$ elements of $X$. For each $p\in \Z_+$, the free vector space consisting of all formal linear combinations of elementary $p$-paths over $X$ with coefficients in $\mathbb{K}$ is denoted $\La_p= \La_p(X)=\La_p(X,\mathbb{K})$. One also defines $\La_{-1}:=\mb{K}$ and $\La_{p}:=\set{0}$ for $p \leq -2$. 
The boundary maps are defined as in Equation (\ref{eq:bd-map}), and we overload notation to denote them by $\p_p$ as before.
It follows that $(\La_p,\p_p)_{p\in \Z}$ is a chain complex.

Next let $G=(X,E)$ be a digraph. For each $p\in \Z_+$, one defines an elementary $p$-path $(x_0,\ldots, x_p)$ on $X$ to be \emph{allowed} if $(x_i,x_{i+1})\in E$ for each $0\leq i\leq p-1$. For each $p\in \Z_+$, the free vector space on the collection of allowed $p$-paths on $(X,E)$ is denoted $\mc{A}_p=\mc{A}_p(G)=\mc{A}_p(X,E,\mathbb{K})$, and is called the \emph{space of allowed $p$-paths}. One further defines $\mc{A}_{-1}:=\mathbb{K}$ and $\mc{A}_{p}:=\set{0}$ for $p \leq -2$.

The allowed paths do not form a chain complex, because the image of an allowed path under $\p$ need not be allowed. This is rectified as follows. Given a digraph $G=(X,E)$ and any $p\in \Z$, the \emph{space of $\p$-invariant $p$-paths on $G$} is defined to be the following subspace of $\mc{A}_p(G)$:
\begin{align*}
\Om_p &=\Om_p(G)=\Om_p(X,E,\mathbb{K})
:=\set{ c \in \mc{A}_p: \p_p(c)\in \mc{A}_{p-1}}. 
\end{align*}
It follows by the definitions that $\im(\p_p(\Om_p))\subseteq \Om_{p-1}$ for any integer $p \geq -1$. Thus we have a chain complex:
\begin{align*}
\ldots \xr{\p_3} \Om_2 \xr{\p_2} \Om_1 \xr{\p_1} \Om_0 \xr{\p_0} \mb{K} \xr{\p_{-1}} 0 
\end{align*}

For each $p\in \Z_+$, the \emph{$p$-dimensional (reduced) path homology groups (denoted $H_p^\Xi$) of $G=(X,E)$} are defined as:
\begin{align*}
H_p^{\Xi}(G)= H_p^{\Xi}(X,E,\mb{K}):=\ker(\p_p)/\im(\p_{p+1}). 
\end{align*}

Note that this definition of path homology is slightly different from the convention in \cite{Grigoryan2014b}, where path homology refers to a version of the above (the non-reduced version) where $\Om_{-1}$ is defined to be $\{0\}$. 

\subsection{Directed flag complex homology}

The \emph{directed flag complex} of a directed graph $G=(X,E)$ is the collection of finite sequences $(x_0,x_1,\ldots,x_n)$, for $n \in \Z_+$, such that $x_i \r x_j$ whenever $i < j$. Such finite sequences are referred to as \emph{directed $n$-simplices}. For each $p \in \Z_+$, we write $\mc{F}:= \mc{F}_p(G)$ to denote the free vector space over directed $p$-simplices in $G$. Then the boundary map $\p$ from Equation (\ref{eq:bd-map}) can be overloaded to give a map $\p_p: \mc{F}_p \r \mc{F}_{p-1}$. The directed flag complex (DFC) homology (denoted $H_p^\mc{F}$) of $G$ is then defined as: 
\begin{align*}
H_p^{\mc{F}}(G)= H_p^{\mc{F}}(X,E,\mb{K}):=\ker(\p_p)/\im(\p_{p+1}). 
\end{align*}

\section{DFC homology of MLPs}


\begin{proof}[Proof of Theorem \ref{thm:dfchmlp}]

First we observe that $\mc{F}_p(\mlp{L}) = \{0 \}$ for each $p \geq 2$, as there are no ``skip connections" from any layer $i$ to a layer $i+j$ for $j \geq 2$. The remainder of the proof will occur at the level of chain complexes, so we introduce some notation for convenience. We write $C^\mc{F}_*$ to denote the chain complex arising from $\mc{F}_*(\mlp{L})$, and $C^\Delta_*$ to denote the chain complex arising from $\Sigma_*(\mlpgraph{L})$. Also let $\p_p^\mc{F}$ denote the boundary map applied to $C_p^\mc{F}$ and let $\p_p^\Delta$ denote the boundary map applied to $C_p^\Delta$. We will typically overload notation and just use $\p$, but the distinction will occasionally be used to clarify context.
It is immediate that $C_p^\mc{F} = C_p^\Delta$ for $p \leq 0$ and $p \geq 2$, so we only verify this equality for $p = 1$.

Note that $C^\mc{F}_1$ is generated by elements of the form $(v,v')$ where $v \r v'$, and $C^\Delta_1$ is generated by elements of the form $[v,v']$ where either $v \r v'$ or $v' \r v$. By using the identity $[v,v'] = -[v',v]$, we write each chain $\s \in C_1^\Delta$ as $\s = \sum_{i=1}^k c_i [u_i,u_i']$, where $u_i \r u_i'$. The crucial point is that in an MLP, we will only have either $v \r v'$ or $v' \r v$, but not both. Thus the chains of $C_1^\Delta$ as written above are exactly the chains of $C_1^\mc{F}$, and the boundary maps are exactly the same for both $C_1^{\mc{F}}$ and $C_1^\Delta$. Thus we have $H_p^\Delta = H_p^{\mc{F}}$ for $p \geq 0$. The second statement follows from standard results on the Euler characteristic of a connected graph. \qedhere

\end{proof}

\section{Path homology of MLPs}

Prior to providing the proof of Theorem \ref{thm:phmlp}, we digress briefly to highlight an interesting connection. The theory of path homology admits K{\"u}nneth formulas for various digraph constructions \cite{Grigoryan2017}, and one might expect that the layered construction of feedforward neural networks would be amenable to applying such formulas. Indeed, when restricted to feedforward network architectures having two layers, the K{\"u}nneth formula for join applies to give the result in Theorem \ref{thm:phmlp}. However, this approach seems not to work for networks with more layers. The immediate obstruction is that feedforward networks with more than two layers do not arise as the join of the individual layers. 

Attempts to prove a K{\"u}nneth formula for a generalized version of digraph join also seem to fail due to the structure of the boundary map $\p_p$. 
Except for this failure (which is not obvious), a simple proof strategy along the following lines would appear convincing at first: introduce an $L$-ary digraph join and appeal to associativity of tensor products (or more elaborately, to an $L$-ary K{\"u}nneth formula \cite{Hungerford1965}) to argue that
\begin{align*}  
&\rank \lp H_p^\Xi(\mlp{L} )\rp \\
&= \sum_{\alpha \in \mathbb{N}^L : \sum_{\ell = 1}^L \alpha_\ell = p-(L-1)} \prod_{\ell = 1}^L \rank \lp H^\Xi_{\alpha_\ell}(\{1,\dots,n_\ell\} )\rp.
\end{align*} 
From here, the desired result would follow from a straightforward calculation.
It would thus be interesting to see if such a K{\"u}nneth formula could provide an alternative proof of Theorem \ref{thm:phmlp}. We remark that K{\"u}nneth formulae in \emph{persistent homology} have been studied in \cite{carlsson2019persistent, gakhar2019kunneth}.

We now proceed to the proof of Theorem \ref{thm:phmlp}.

\begin{lemma}[Product rule, \cite{Grigoryan2017}] 
\label{lem:prod}
Let $u \in \La_p$ and $v \in \La_q$. Then: 
\[ \p(uv) = (\p u)v + (-1)^{p+1}u(\p v).\]
\end{lemma}

\noindent
\textbf{Notation.} We adopt some extra notation for readability. Given (the underlying digraph of) an MLP $\mlp{L}$, we write $\mc{A}^L_p$ to denote $\mc{A}_p(\mlp{L})$. We define $\Om_p^L$ and $\p_p^L$ analogously. Recall that $K_i$ denotes the $i$th layer of $\mlp{L}$.

To prove the theorem, we need to understand the kernel of $\p_{L-1}^L$. The next proposition gives a representation of each such kernel element in terms of an $(L-1)$-layer MLP.

\begin{proposition} 
\label{prop:repn-sum}
Consider the map $\p_{L-1}^L$ defined on $\Om_{L-1}^L$. 
Any element $\g \in \ker(\p_{L-1}^L)$ can be written as a finite sum
\[ \g = \sum_{i=1}^d w_i v_i, \qquad w_i \in \ker(\p_{L-2}^{L-1}),\, v_i \in K_L, \, d \geq 1.\]
\end{proposition}

\begin{proof}
By the structure of $\mlp{L}$, all $(L-1)$-paths have the form $(v^{(1)}v^{(2)}\ldots v^{(L)})$, where each $v^{(i)} \in K_i$. Thus any element in $\ker(\p_{L-1}^L)$ has the form 
\[ \g = \sum_{i=1}^d w_iv_i \in \Om_{L-1}(\mlp{L}), \qquad \text{$w_i \in \mc{A}^{L-1}_{L-2}$, $v_i \in K_L$}.\]
Here we assume WLOG that the $v_i$ are distinct.
By Lemma \ref{lem:prod}:
\begin{align*}
0 = \p (\g) = \sum_{i=1}^d (\p w_i)v_i + (-1)^{L-1} w_i. 
\end{align*}
By linear independence, we must have $(-1)^{L-1}\sum_{i=1}^d w_i = 0$ and $\sum_{i=1}^d   (\p w_i)v_i = 0$. 
Since the $v_i$ are all distinct, we have $\p w_i = 0$ and hence $w_i \in \ker(\p_{L-2}^{L-1})$ for each $1\leq i \leq d$. This concludes the proof. \qedhere 

\end{proof}

The next proposition further clarifies the preceding representation as a difference of basis terms. 

\begin{proposition} 
\label{prop:repn-diff}
Consider the map $\p_{L-2}^{L-1}$ defined on $\Om_{L-2}^{L-1}$, and let $B_{L-2}$ be a basis for $\ker(\p_{L-2}^{L-1})$. 
Let $\g \in \ker(\p_{L-1}^L)$.  Then we can write
\[ \g = \sum_{i=1}^{|B_{L-2}|} \sum_{j,k=1}^{|K_L|}  c_{ijk} u_i (v_j - v_k),\]
where $c_{ijk} \in \mb{K}$, $u_i \in B_{L-2}$, and $v_j,v_k  \in K_L$.
\end{proposition}

\begin{proof} 
Using Proposition \ref{prop:repn-sum}, we write $\g = \sum_{i=1}^d c_i u_i v_i$, where $d \geq 1$, $c_i \in \mb{K}$, $u_i \in B_{L-2}$, and $v_i \in K_L$. This allows for degeneracy, in the sense that we may have $u_i = u_j$ for $i \neq j$, and likewise for $c_i$ and $v_i$. By Lemma \ref{lem:prod}, we obtain $0 = \p(\g) = \sum_{i=1}^d c_i u_i$. Here we have used the relation $\p(u_i) = 0$. 

Next fix $i =1$, and let $I_1 \subseteq \{1,2,\ldots, d\}$ denote the indices $j$ for which $u_j = u_1$. Since $\sum_{i=1}^d c_i u_i =0$, we have by linear independence that $\sum_{j \in I_1} c_j = 0$. 
Next let $I_1^+ \sqcup I_1^-$ be a partition of $I_1$ into nonempty sets. Then we have $\sum_{j \in I_1^+}c_j = \sum_{k \in I_1^-}-c_k$ by the preceding observation. It follows that
\[ \sum_{j \in I_1} c_j u_j v_j = \sum_{j \in I_1} c_j u_1 v_j = \sum_{j \in I_1^+} c_j u_1 v_j - \sum_{k \in I_1^-} (-c_k) u_1 v_k.\] 
We would like to write the latter as a sum of elements of the form $c_{jk} u_1 (v_j - v_k)$. 
The problem of determining the coefficients $c_{jk}$ can be phrased as a supply-demand problem, and we present this next. For now we assume $\sum_{j \in I_1^+}c_j = \sum_{k \in I_1^-}-c_k \neq 0$. 

Define the \emph{supply} vector $r$ to be an $|I_1^+| \times 1$ column vector with entries $c_j$, $j \in I_1^+$. Also define the \emph{demand} vector $s$ to be the $1 \times |I_1^-|$ row vector with entries $-c_k$, $k \in I_1^-$. We need to construct a nonnegative $|I_1^+| \times |I_1^-|$ matrix $T^1$ with row and column sums equal to $r$ and $s$, respectively. Define $T^1$ by writing $T^1_{jk} := r_js_k/ \sum_{l \in I_1^+}c_l$ for each $j,k$.

To verify that $T^1$ has the desired row and column sums, recall that $\sum_{k \in I_1^-} s_k = \sum_{k \in I_1^-}-c_k = \sum_{j \in I_1^+} c_j$, and so
\[\sum_{k \in I_1^-} T^1_{jk} = r_j \sum_{k \in I_1^-} s_k/ \sum_{l \in I_1^+}c_l =  r_j \sum_{j \in I_1^+} c_j/ \sum_{l \in I_1^+}c_l = r_j.\]
Similarly, recall $\sum_{j \in I_1^+} r_j = \sum_{j \in I_1^+} c_j$, and so
\[\sum_{j \in I_1^+} T^1_{jk} = s_k \sum_{j \in I_1^+} r_j/ \sum_{l \in I_1^+}c_l =  s_k \sum_{j \in I_1^+} c_j/ \sum_{l \in I_1^+}c_l = s_k.\]
Thus we obtain:
\[ \sum_{j \in I_1^+} c_j u_1 v_j - \sum_{k \in I_1^-} (-c_k) u_1 v_k = \sum_{j \in I_1^+}\sum_{k \in I_1^-} T^1_{jk}u_1(v_j - v_k).\]
It may be the case that $v_j = v_{j'}$ for $j \neq j' \in I_1^+$, and likewise for $I_1^-$. By summing such terms and by padding $T^1$ with zeros if necessary, we create a $|K_L| \times |K_L|$ matrix $C^1$ satisfying:
\begin{align} 
\sum_{j \in I_1^+}\sum_{k \in I_1^-} T^1_{jk}u_1(v_j - v_k) = \sum_{j,k = 1}^{|K_L|} C^1_{jk} u_1 (v_j - v_k). \label{eq:OT}
\end{align}

Now we return to the case $\sum_{j \in I_1^+}c_j = \sum_{k \in I_1^-}-c_k = 0$. In this case, we further subdivide $I_1^+$ into nonempty sets $I_1^{++} \sqcup I_1^{+-}$. If $\sum_{j \in I_1^{++}}c_j = \sum_{k \in I_1^{+-}}-c_k \neq 0$, then we can proceed as before to obtain a decomposition as in Equation (\ref{eq:OT}), and otherwise we can continue subdividing the index set. Since there are only finitely many terms, the subdivision operation must terminate in a finite number of steps. Similarly one subdivides $I_1^-$ as necessary to collect terms in the form of Equation (\ref{eq:OT}). 

Repeating this process for $i = 2,\ldots, d$, we obtain:
\[ \g =  \sum_{i=1}^d c_i u_i v_i = \sum_{i=1}^{|B_{L-2}|} \sum_{j,k =1}^{|K_L|} C^i_{jk} u_i(v_j - v_k). \qedhere\]

\end{proof}

\begin{proposition}
\label{prop:repn-1-diff}
Given the setup of Proposition \ref{prop:repn-diff} and $\g \in \ker(\p_{L-1}^L)$, we can further write:
\[ \g = \sum_{i=1}^{|B_{L-2}|} \sum_{j = 2}^{|K_L|}  c_{ij} u_i (v_1 - v_j),\]
where $c_{ij} \in \mb{K}$, $u_i \in B_{L-2}$, and $v_j  \in K_L$.
Consequently, $B_{L-1}:= \set{ u_i(v_1 - v_j) : u_i \in B_{L-2}, \, 2\leq j \leq |K_L|}$ forms a basis for $\ker(\p_{L-1}^L)$. 
\end{proposition}

\begin{proof}
We repeat the first few steps of Proposition \ref{prop:repn-diff}; namely we use Proposition \ref{prop:repn-sum} to write $\g = \sum_{i=1}^d c_i u_i v_i$, observe $\sum_{i=1}^d c_i u_i = 0$, and collect $u_i$ terms to write 
\[ \g =  \sum_{i=1}^{|B_{L-2}|} u_i \sum_{j=1}^{|K_L|} c_{ij} v_j.\]
We further write
\begin{align*}
 \g =\sum_{i=1}^{|B_{L-2}|} u_i \lp \sum_{j=1}^{|K_L|} c_{ij} (v_j - v_1) + \sum_{j=1}^{|K_L|} c_{ij} v_1 \rp
\end{align*}
and observe that, by the preceding observation:
\[ \sum_{i=1}^{|B_{L-2}|} u_i \sum_{j=1}^{|K_L|} c_{ij} v_1 = \sum_{i=1}^d c_i u_i v_1 =0. \qedhere\]

\end{proof}

\begin{proposition} 
\label{prop:2lp}
Let $K$ be a two-layer MLP, and consider the boundary map $\p_1^2$ on $\Om_1(K)$. Let $u_1,u_2,\ldots, u_{n_1}$ and $v_1,v_2,\ldots, v_{n_2}$ denote the vertices of the first and second layers of $K$, respectively. Then $\ker(\p_1^2)$ is generated by the elements $\set{(u_1 - u_j)(v_1 - v_k) : 2\leq j \leq n_1,\, 2\leq k \leq n_2}$. In particular, $\dim(\ker(\p_1^2)) = (n_1 - 1)(n_2 - 1)$.
\end{proposition}

\begin{proof}
Let $K_1$ denote the first layer of $K$, and recall that $\p_0^1$ denotes the boundary map defined on $\Om_0(K_1)$. Since we are computing reduced homology, $\ker(\p_0^1)$ is generated by elements of the form $u_j - u_k$. We rewrite this as $u_j - u_k = u_j - u_1 + u_1 - u_k = -(u_1 - u_j) + (u_1 - u_k)$. Thus a basis for $\ker(\p_0^1)$ is given by $\set{u_1 - u_j : 2\leq j \leq |K_1|}$. An application of Proposition \ref{prop:repn-1-diff} completes the proof.
\end{proof}

\begin{proposition}
\label{prop:hom-vanish}
Given an MLP $\mlp{L}$ with $L$ layers, we have $\ker(\p^L_{j}) = \im(\p^L_{j+1})$ for each $0\leq j \leq L-2$.  
\end{proposition}

\begin{proof} 
Let $0 \leq j \leq L-2$, and let $\g \in \ker(\p^L_j)$. Then $\g$ is a linear combination of paths of length $(j+1)$. The endpoints of these paths may belong to layers $K_{j+1}, \ldots, K_L$. By collecting paths ending at the same layer, write $\g = \sum_{l=j+1}^L \g_l$, where each $\g_l$ consists of the summands of $\g$ ending at $K_l$ (and is 0 if there are no such summands). By linear independence, we must individually have $\p(\g_l) =0$ for each $k+1 \leq l \leq L$.  

Let $k+1 \leq l < L$, and let $z_{l+1} \in K_{l+1}$. By Lemma \ref{lem:prod},  
\begin{align*}
\p(\g_l z_{l+1}) = (\p \g_l) z_{l+1} + (-1)^{j+1}\g_l(\p z_{l+1}) = (-1)^{j+1} \g_l.
\end{align*}

Next fix $l = L$, and note that $\g_L$ is a linear combination of paths that start at $K_{L-j}$ and end at $K_L$. Fix $z' \in K_{L- (j+1)}$. Then we have: 
\begin{align*}
\p(z' \g_L) = (\p z')\g_L + (-1)z'(\p\g_L) = \g_L.  
\end{align*}

Finally define $\zeta = \sum_{l = j+1}^{L-1} (-1)^{j+1} \g_l z_{l+1} + z'\g_L.$ By the previous work, we have $\p(\zeta) = \sum_{l = j+1}^L \g_l = \g$. By construction, $\zeta \in \Om^L_{j+1}$. This shows that $\ker(\p^L_{j}) = \im(\p^L_{j+1})$ and concludes the proof.\qedhere

\end{proof}

Now we proceed to the proof of Theorem \ref{thm:phmlp}.

\begin{proof} 
$\mlp{L}$ has no $j$-paths for $j \geq L$. Thus by Proposition \ref{prop:hom-vanish}, there can only be nontrivial reduced path homology in degree $L-1$. Applying Proposition \ref{prop:repn-1-diff} inductively while using Proposition \ref{prop:2lp} as a base case, we see that $\ker(\p_{L-1}^L)$ has dimension $\prod_{i=1}^L (n_i - 1)$. Since there are no $L$-paths, $\im(\p_{L}^L)$ is trivial. The result follows.\qedhere

\end{proof}

\begin{figure}[t]
\includegraphics[trim = 60mm 130mm 55mm 125mm, clip, width=.24\textwidth,keepaspectratio]{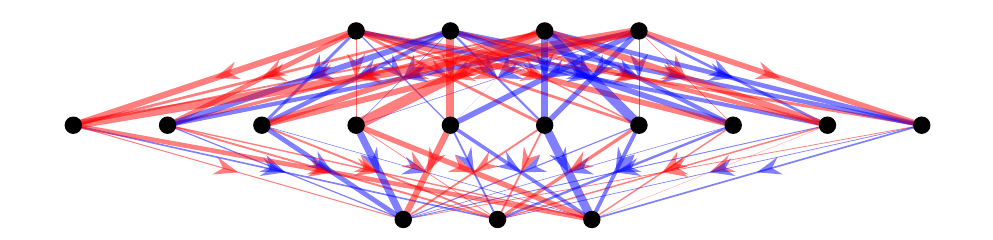}
\includegraphics[trim = 60mm 130mm 55mm 125mm, clip, width=.24\textwidth,keepaspectratio]{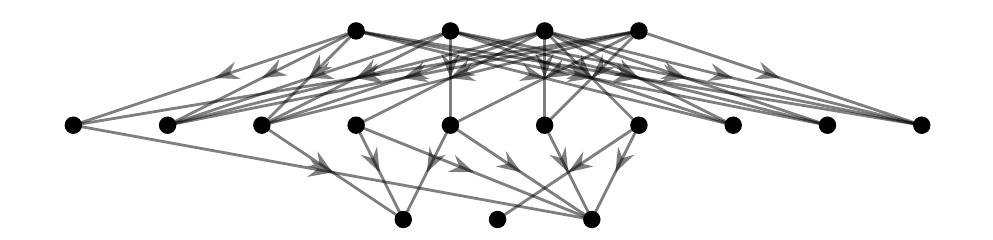}
\centering
\caption{(L) A MLP with $L = 3$ layers and $(n_1,n_2,n_3) = (4,10,3)$, with weight magnitudes indicated by arc thickness and signs indicated by color (red = negative; blue = positive). (R) A subgraph obtained by removing weights with magnitude below the median.}
\label{fig:IrisMLP}
\end{figure}

\begin{figure}[t]
\includegraphics[trim = 45mm 85mm 45mm 85mm, clip, width=.45\textwidth,keepaspectratio]{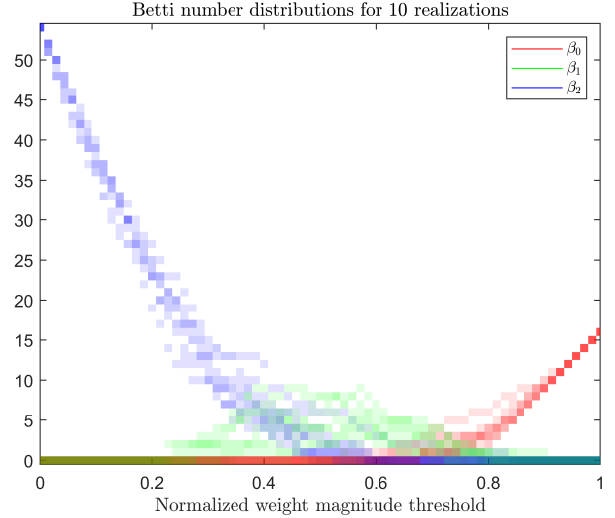}
\centering
\caption{Path homology Betti numbers $\beta_p$ of a trained 3-layer fully-connected network ($(n_1,n_2,n_3) = (4,10,3)$) across normalized weight magnitude thresholds and for 10 realizations of random initial weights. The distribution of Betti numbers is indicated by opacity. The threshold normalization sends the $j$th smallest threshold value to $j/T$, where $T$ is the number of nontrivial threshold values. Note that the network starts with path homology concentrated in degree 2, as predicted by Theorem \ref{thm:phmlp}. Surprising features include the gradual downward cascade of $\b_2$ as well as the ``bump" in $\b_1$.}
\label{fig:dnn_thresholds}
\vspace{-.5cm}
\end{figure}

\section{Discussion}

In this work, we provide characterization results for two types of digraph homologies as applied to feedforward neural network architectures. Our results show that these two homology theories, while similar in structure, yield quite different outputs (with different interpretations) when applied to deep networks. From this perspective, it is important to utilize both types of digraph homology when studying neural architectures.

As an example of the utility of path homology to characterize neural networks, we used Fisher's classical iris data set \cite{Fisher1936} to train a MLP with $L = 3$ layers and $(n_1,n_2,n_3) = (4,10,3)$ using MATLAB's \texttt{patternnet} function. We performed multiple training runs with different realizations of random initial conditions. For each realization, we extracted the trained weight matrix $A$. We then computed the path homology of the DAG obtained by removing arcs corresponding to weights less than a given nontrivial value in $A$ (details of this and other experimental results obtained using the path homology algorithm and its implementation will be provided elsewhere). The results are shown in Figure \ref{fig:dnn_thresholds}. The figure shows that the path homology of a filtered subgraph of a fully-connected network does not suddenly vanish as edges are removed. Rather, the Betti numbers gradually ``cascade down'' from top to zero dimension. This phenomenology suggests that path homology can give particularly detailed topological characterizations of neural networks. For example, given a weight matrix, the corresponding filtered (or persistent) path homology of subnetworks obtained by restricting attention to a moving window of a few adjacent layers can help identify notional sub-networks that exhibit functional specificity.

With further homological characterization of directed graph motifs like those presented in this paper, persistent homological structure derived empirically from analyses of network parameters can be related back to network architectures. This link between parameterized neural network topology and neural network architecture provides actionable insight in the architectural design process and can constrain the architecture search space, as one can  define architectures that \textit{a priori} better suit the observed homological structure learned by a network for a given task.

\textbf{Acknowledgements.} We would like to thank Guilherme Vituri, Michael Robinson, and Bernard McShea for useful discussions. We are very grateful to the anonymous reviewers for their useful suggestions.

\bibliographystyle{ieeetr}
\bibliography{pathHomology}

\end{document}